\documentclass[a4paper,11pt]{article}

\usepackage{amsmath}
\usepackage{amsmath,amsthm,amscd,amssymb,bbm,graphicx}
\usepackage{floatflt,setspace, wrapfig}

\newtheorem{thm}{Theorem}[section]
\newtheorem{lemma}[thm]{Lemma}
\newtheorem{prop}[thm]{Proposition}
\newtheorem{cor}[thm]{Corollary}

\newtheorem{defn}[thm]{Definition}
\newtheorem{examp}[thm]{Example}

\newtheorem{rmk}[thm]{Remark}

\newcommand{\N}{\mathbb{N}}

\newcommand{\mR}{\mathbb{R}}

\newcommand{\var}{\mbox{var}}

\newcommand{\diam}{\mbox{diam}}

\begin{document}

\title{ Erd\H{o}s-R\'enyi laws for exponentially and polynomially mixing dynamical systems.}

\author{Nicolai Haydn and Matthew Nicol
\thanks{Department of Mathematics,
University of Southern California; Department of Mathematics, University of Houston.
 E-mail: $<$nhaydn@usc.edu$>$,$<$nicol@math.uh.edu$>$. 
 MN would like to thank the  NSF for  support on NSF-DMS Grant 2009923. }
 }

\date{\today}

\maketitle

\begin{abstract}
Erd\H{o}s-R\'enyi limit laws give the  length scale of a time-window over which time-averages in Birkhoff sums have a
non-trivial almost-sure limit. We establish Erd\H{o}s-R\'enyi type limit laws for H\"older observables on dynamical systems modeled
by Young Towers with exponential   and polynomial  tails. This extends earlier results on Erd\H{o}s-R\'enyi limit laws 
to a broad class of dynamical systems with some degree of hyperbolicity. 
\end{abstract}

\section{Introduction}\label{sec:intro}

The Erd\H{o}s-R\'enyi  fluctuation law gives the length scale of a time-window over which time-averages in Birkhoff sums have a non-trivial almost-sure limit. 
It  was first proved in the independent and 
identically distributed  (i.i.d.) case~\cite{Erd}  in the following form:

\begin{prop}\label{ER} Let $(X_n)_{n\ge 1}$ be an i.i.d.\  sequence of non-degenerate random variables, 
$\mathbb{E}[X_1]=0$, and
 let $S_n=X_1+\cdots+X_n$. Assume that the moment generating function $\phi(t)= \mathbb{E}(e^{tX_1})$ exists in some open  interval $U\subset \mR$  containing $t=0$. 
For each $\alpha>0$, define $\psi_\alpha(t)= \phi(t) e^{-\alpha t}$. For those $\alpha$ for which $\psi_\alpha$ attains its minimum at a point $t_\alpha\in U$, let $c_{\alpha}=\alpha t_{\alpha} -\ln \phi (t_{\alpha})$. Then

\[
 \lim_{n\to\infty} \max\{ (S_{j+[\ln  n /c_{\alpha}}-S_j)/[\ln  n/c_{\alpha}]: 1\le j\le n-[\ln  n/c_{\alpha}]\}=\alpha
 \]

\end{prop}

The existence of $\psi_{\alpha} (t)$
for all $t \in U$ implies exponential large deviations with a rate function (in fact $c_{\alpha}=I(\alpha)$ where
$I$ is the rate function, defined later) and  this implies that sampling over a window length $k(n)$ of larger than logarithmic length scale (in the sense that $k(n)/\ln  n \to \infty$), allows  the ergodic theorem to kick in and 
\[
 \lim_{n\to\infty} \max\{ (S_{j+k(n)}-S_j)/k(n): 1\le j\le n-k(n)\}=0
 \]
while sampling over too small a window, for example  $k(n)=1$, gives similarly a trivial limit
\[
 \lim_{n\to\infty} \max\{ (S_{j+k(n)}-S_j)/k(n): 1\le j\le n-k(n)\}=\|X_1\|_{\infty}
 \]

Define the function
$$
\theta (n,k(n)):= \max_{0\le j\le n-k(n)}\frac{S_{j+k(n)}-S_j}{k(n)},
$$
which  may be interpreted as the maximal average gain over a time window of length
 $k(n)$ up to time $n$.
In the  setting of coin tosses  
the Erd\H{o}s-R\'enyi law gives precise 
information on the maximal average gain of  a player in a fair game  in the case where the length of the time window
ensures  $\lim_{n\to \infty} \theta (n,k(n))$ has a non-degenerate almost sure  limit. 




In 1986 Deheuvels, Devroye and Lynch~\cite{Deh}  in the i.i.d.\ setting of Proposition~\ref{ER} gave a precise rate of convergence and showed that if
$k(n)=[\ln n/c_{\alpha}]$ then $P$~a.s:
\[
\limsup\frac{[\theta (n,k(n))-\alpha k(n)]}{\ln k(n)}=\frac{1}{2t_{\alpha}}
\]
and 
\[
\liminf\frac{[\theta (n,k(n))-\alpha k(n)]}{\ln k(n)}=-\frac{1}{2t_{\alpha}}
\]

In this paper we  establish Erd\H{o}s-R\'enyi limit laws  for H\"older observables on dynamical systems
modeled by  Young Towers~\cite{LY98,LY99} with exponential and 
polynomial tails. Tails refer to the measure $\mu( R>n)$ of the return time $R$  function to the base of  the tower.  Our exposition 
is based upon~\cite[Section 2.3]{KKM}  and~\cite{Melbourne_Varandas} who present a framework more general than that of the original Tower construction of Young~\cite{LY98} in that uniform 
contraction of local stable manifolds is not assumed for polynomially mixing systems in dimensions greater than $1$. We will give more details on Young Towers below but here note that H\"older observables on Young Towers with exponential (polynomial) tails
have exponential (polynomial) decay of correlations, the precise rate is encoded in the return time function.

Our results extends the work
of~\cite{Nic} from the class of non-uniformly expanding maps with exponential decay of correlations to all systems modeled by 
a Young Tower, including Sinai dispersing billiard maps; diffeomorphisms of Hen\'{o}n type; polynomially
mixing billards as in~\cite{Chernov_Zhang2} (as long as the correlation decay rate is greater than $n^{-\beta}$, $\beta>1$);
smooth unimodal and multimodal maps satisfying the Collet-Eckmann conditions~\cite[Example 4.10]{KKM};
certain Viana maps~\cite[Example 4.11]{KKM};  and Lorenz-like maps. Other examples to which our results apply
are listed in~\cite{Melbourne_Varandas}.

 In the setting of hyperbolic dynamical systems there are many earlier results.  
 Grigull~\cite{Gri} established the Erd\H{o}s-Renyi law for hyperbolic rational maps, Chazottes and Collet~\cite{Col}  proved Erd\H{o}s-Renyi theorems with rates for uniformly expanding maps of the interval, while
 Denker and Kabluchko~\cite{Den1} proved Erd\H{o}s-Renyi results for Gibbs-Markov dynamics. In~\cite{Denker}
   Erd\H{o}s-R\'enyi limit laws for Lipschitz observations on  a class of non-uniformly expanding dynamical systems, including logistic-like maps, were given as well as related results on maximal averages of a time series arising from  H\"older observations
on intermittent-type maps
over a time window of polynomial
 length.    Kifer~\cite{Kifer1,Kifer2} has established Erd\H{o}s-R\'enyi laws for non-conventional ergodic sums and in the setting of averaging or homogenization of chaotic dynamical systems. We mention also recent related work of~\cite{Chen1,Chen2} on  applications of Erd\H{o}-Reny\'i limit laws to multifractal analysis.

 The main novelty of out technique is the use of the symbolic metric  on the axiomatic Young Tower construction of~\cite{Melbourne_Varandas, KKM} to control the  norm of the indicator function of sets of the form $(S_n > n\alpha)$ on the quotiented tower. This eliminates many difficulties involved with considering the Lipschitz 
 norm of such sets with respect to the Riemannian metric on  the phase space of the system. The structure allows
 us to consider, with small error, averaged Birkhoff sums as being constant on stable manifolds, and thence
 use the decay of correlations for observables on  the quotiented tower in terms of their Lipschitz and $L^{\infty}$ norms. 
 
 Our results in the case of Young Towers with exponential decay of correlations, Theorem~\ref{main}, are optimal and replicate the i.i.d case, while in the case of Young Towers with polynomial tails
 we investigate windows of polynomial length and give close to optimal upper and lower bounds, Theorem~\ref{poly1} and Theorem~\ref{poly2}.

\section{Young Towers.}\label{sec-NUH}

We now describe more precisely what we mean by a non-uniformly hyperbolic dynamical system modeled by a Young Tower.  Our exposition 
is based upon~\cite[Section 2.3]{KKM}  and~\cite{Melbourne_Varandas} who present a framework more general than that of the original Tower of Young~\cite{LY98} in that uniform 
contraction of local stable manifolds is not assumed for polynomially mixing systems in dimensions greater than $1$. This set-up is very useful for the study of almost sure fluctuations of
Birkhoff sums of bounded variables.

We suppose $T$ is a diffeomorphism of a Riemannian manifold $(M,d)$, possibly with singularities. Fix a subset $\Lambda\subset M$ with a 
`product structure'. Product structure means there exists a family of disjoint stable disks (sometimes called local stable manifolds)  $\{W^s\}$ that cover $\Lambda$ as well as a family
of disjoint unstable disks (sometimes called local unstable manifolds)  $\{W^u\}$ that cover $\Lambda$. The stable and unstable 
disks containing $x\in \Lambda$ are denoted $W^s (x)$ and $W^u (x)$. Each stable disk intersects each unstable disk in precisely one point. 

Suppose there is a partition $\{\Lambda_j\}$ of $\Lambda$ such that each  stable disk $W^s (x)$  lies in $\Lambda_j$ if $x\in \Lambda_j$.
Suppose there exists  a `return time' integer-valued function $R:\Lambda \to \N$, constant with value $R(j)$
on each partition element $\Lambda_j$, such that $T^{R(j)}(W^s (x)) \subset W^s (T^{R(j)} x)$ for all $x\in \Lambda_j$. We assume that
the greatest common denominator of the integers  $\{ R(j)\}$ is $1$, which ensures that the Tower is mixing.   We define the induced return
map $f: \Lambda \to \Lambda$ by $f(x)=T^{R(x)} (x)$. 

For $x,y \in \Lambda$ let $s(x,y)$ be the least integer $n\ge 0$ such that $f^n (x)$ and $f^n (y)$ lie in different partition elements of $\Lambda$. We call
$s$ the separation time with respect to the map $f: \Lambda \to \Lambda$. 

\noindent {\bf Assumptions:}   there exist constants $K\ge 1$ and $0<\beta_1<1$ such that 

(a) if $z \in W^s (x)$ then $d(f^n z, f^n x) \le K \beta_1^n$;

(b)  if $z\in W^u (x)$ then  $d(f^n z, f^n x)\le K \beta_1^{s(x,z)-n}$;

(c) if $z,x\in \Lambda$ then $d(T^j z,T^j x) \le K (d(z,x)+d(fz, fx))$ for all $0\le j \le \min \{ R(z), R(x)\}$. 

Define an equivalence relation on $\Lambda$ by $z\sim x$ if $z\in W^s (x)$ and form the quotient space  $\overline{\Lambda}=\Lambda/\sim$
with corresponding partition $\{ \overline{\Lambda_j}\}$. The return time function $R: \overline{\Lambda}\to \N $ is well-defined as 
each  stable disk $W^s (x)$  lies in $\Lambda_j$ if $x\in \Lambda_j$ and $T^{R(j)}(W^s (x)) \subset W^s (T^{R(j)} x)$ for all $x\in \Lambda_j$.
So we have a well-defined induced map $\bar{f}: \overline{\Lambda} \to \overline{\Lambda}$. Suppose that $\bar{f}$ and the partition $\{\overline{\Lambda_j}\}$ 
separates points in $\overline{\Lambda}$. 
 Define $d_{\beta_1} (z,x)=\beta_1^{s(z,x)}$, then $d_{\beta_1}$  is
a metric on $\overline{\Lambda}$.

 Let $m$ be a reference probability measure on $\overline{\Lambda}$ (in most applications this will be normalized Lebesgue measure).
 Assume that $\bar{f}:\overline{\Lambda}\to \overline{\Lambda}$ is a Gibbs-Markov uniformly expanding on $(\overline{\Lambda}, d_{\beta_1})$. By this we mean that 
$\bar{f}$ is a measure-theoretic bijection from each $\overline{\Lambda_j}$ onto 
$\overline{\Lambda}$. 


 We assume that $\bar{f}: \overline{\Lambda} \to \overline{\Lambda}$ has an invariant probability measure $\overline{\nu}$ and $0<a<\frac{d\bar{\nu}}{dm} < b$ for some constants $a,b$. We 
assume that $R$ is $\overline{\nu}$-integrable and there is an $f$ invariant probability $\nu$  measure on $\Lambda$ such that $\overline{\pi}^{*} \nu=\overline{\nu}$ where
$\overline{\pi}$ is the quotient map taking $\Lambda$ onto $\Lambda/\sim$.  Now we define the Young Tower
\[
\Delta =\{ x,j)\in \Lambda \times \N: 0\le j \le R(x)-1\}
\]
and  the tower map $F$ by
\[ 
F(x,j)= \left\{ \begin{array}{ll}
         (x,j+1) & \mbox{if $j < R(x)-1$};\\
        (fx,0)& \mbox{if $j=R(x)-1$}.\end{array} \right. 
\] 
and lift $\nu$ in a standard way  to  an invariant probability  measure $\nu_{\Delta}$ for $F: \Delta \to \Delta$. In fact $\nu_{\Delta}=\nu \times \mbox{counting measure}$.

Define the semi-conjugacy $\pi : \Delta \to M$, 
$\pi (x,j)=T^j (x)$. The measure $\mu=\pi^{*} \nu_{\Delta}$ is a $T$-invariant mixing probability measure on $M$. Given  an observable $\varphi: M \to \mathbb{R}$ we may lift to
an observable $\varphi: \Delta \to \mathbb{R}$  by defining $\varphi (x,j)=\varphi (T^j x)$ (we keep 
the same notation for the observable).   The semi-conjugacy $\pi^{*}$ allows us to transfer statistical properties from 
lifted observables $\varphi$ on $(\Delta, F, \nu_{\Delta} )$ to the original observables $\varphi$ on $(T, M, \mu)$.

\section{Large deviations and rate functions.}
      
   Before stating precisely our main result we recall the definition of rate function and some other notions of large deviations theory.    
  Suppose $(T, M,\mu)$ is  a probability preserving transformation and $\varphi: M \to \mathbb R$ is a mean-zero integrable function i.e.\
$\int_M \varphi~d\mu =0$. 
Throughout this paper we will write $S_n (\varphi ): =\varphi +\varphi\circ T +\ldots + \varphi \circ T^{n-1} $
for the $n$th ergodic sum of $\varphi$. Sometimes we will write $S_n$ instead
of $S_n(\varphi)$ for simplicity of notation or when $\varphi$ is clear from context.

\begin{defn}  A mean-zero integrable function $\varphi: M \to \mathbb R$  is said to satisfy a large deviation principle with rate function $I(\alpha)$, if there exists a non-empty neighborhood $U$ of $0$ and a strictly convex function $I:U\to \mathbb R$, non-negative and vanishing only at $\alpha=0$, such that
\begin{eqnarray}\label{rate+}
\lim_{n\to\infty} \frac 1n\log \mu (S_n (\varphi)  \ge n \alpha)& =& -I(\alpha)
\end{eqnarray}
for all $\alpha>0$ in $U$ and
\begin{eqnarray}\label{rate-}
 \lim_{n\to\infty} \frac 1n\log \mu (S_n (\varphi)  \le n \alpha) &=& -I(\alpha)
 \end{eqnarray}
for all $\alpha<0$ in $U$.
\end{defn}

In the literature this is referred to as a first level or local (near the average) large deviations principle.

For  H\"older observables on Young Towers with exponential tails (which are not $L^1$ coboundaries in the sense that 
$\varphi\not = \psi\circ T-\psi$
   for any $\psi\in L^1 (\mu)$) such an exponential large deviations result holds
with rate function $I_{\varphi} (\alpha)$~\cite{Nic,Reybellet_Young,Mel,Pollicott_Sharp}.  A formula  for the width
of $U$ is given in \cite{Reybellet_Young} following a standard approach but it is not useful
in concrete estimates. 

\section{Erd\H{o}s-R\'enyi laws: background.}\label{erdoslaw}

Proposition~\ref{prop:erdos1} given below  is found in  a proof from Erd\H{o}s and R\'enyi~\cite{Erd} (see  \cite[Theorem 2.4.3]{Csorgo_Revesz}, Grigull~\cite{Gri} Denker and  Kabluchko~\cite{Den1} or~\cite{Denker} where this method has been used). 
The Gauss bracket $[.]$ denotes the integer part of  a number. 
 Throughout the proofs of this paper we will concentrate on the case $\alpha>0$
as the case $\alpha<0$ is identical with the obvious modifications of statements.

\begin{prop}\label{prop:erdos1}
Let $(T,M,\mu)$ be an ergodic dynamical system and $\varphi: M \to \mathbb R$ is an observable.

(a) Suppose that $\varphi$ satisfies a large deviation principle with rate  function $I$ defined on the open set $U$ and assume $\mu(\varphi)=0$ Let $\alpha >0$, $\alpha \in U$ and set 
$$ L_n=L_n(\alpha)=\left[\frac{\ln n}{I(\alpha)}\right]\qquad n\in\mathbb N.$$
Then the upper Erd\H{o}s-R\'enyi law holds, that is, for $\mu$ a.e. $x\in X$ 
$$
 \limsup_{n\to\infty} \max_{0\le j\le n-L_n}\frac1{L_n}S_{L_n} (\varphi) \circ T^j (x)\le \alpha.
 $$

\noindent (b) If  for some constant $C>0$ and integer $\kappa\ge 0$ for each interval $A$
\begin{eqnarray}\label{tau}
 \mu\!\left(\bigcap_{m=0}^{n-L_n}\{S_{L_n} (\varphi) \circ T^m\in A\}\right)&\le &C [\mu (S_{L_n}\in A)]^{n/(L_n)^{\kappa}}
 \end{eqnarray}
then the lower Erd\H{o}s-R\'enyi law holds as well, that is, for $\mu$  a.e. $x\in X$ 
$$ 
\liminf_{n\to\infty} \max_{0\le j\le n-L_n}\frac1{L_n}S_{L_n} (\varphi) \circ T^j \ge \alpha.
$$
\end{prop}

\begin{rmk}
If both Assumptions (a) and (b) of Proposition~\ref{prop:erdos1} hold then 
\[
\lim_{n\to \infty}  \max_{0\le m\le n-L_n} \frac{S_{L_n}\circ T^m}{L_n}=\alpha.
\]
\end{rmk} 

\begin{rmk}\label{rem:sum}
The proof of the lemma shows that the upper Erd\H{o}s-R\'enyi law follows from the existence of 
exponential large deviations given by a rate function, while for the  lower Erd\H{o}s-R\'enyi law it suffices to show that  for every $\epsilon>0$ 
the series $\sum_{n>0} \mu (B_n (\epsilon))$, where $B_n(\epsilon)=\{\max_{0\le m\le n-L_n} S_{L_n}\circ T^m \le L_n(\alpha-\epsilon)\}$ is summable. This is usually the harder part to prove in the deterministic case.
\end{rmk}

   \section{Erd\H{o}s-R\'enyi limit laws for Young Towers with exponential tails.}
   
  We now state our main theorem in the case of exponential tails. 
   \begin{thm}~\label{main}
Suppose $(T,M,\mu)$ is a dynamical system modeled by a Young Tower with $\nu_{\Delta} (R>j) \le C \beta_2^j$ for some $\beta_2 \in (0,1)$ and 
some constant $C_2$. Let $\varphi: M \to \mR$ be H\"older with $\int \varphi~d\mu=0$.  Assume $\varphi\not = \psi\circ T-\psi$
for any $\psi\in L^1 (\mu)$.  Let $I(\alpha)$ denote the non-degenerate  rate function  defined on an open set $U\subset \mR$ containing $0$.  Define $S_n (x) =\sum_{j=0}^{n-1}\varphi (T^j x)$.

  Let $\alpha >0$, $ \alpha\in U$ and  define 
\[
L_n=L_n(\alpha)=\left[\frac{\ln n}{I(\alpha)}\right]\qquad n\in\mathbb N.
\]
Then
$$
\lim_{n\to \infty}  \max_{0\le j\le n-L_n} \frac{S_{L_n}\circ T^j (x)}{L_n}=\alpha,
$$
for $\mu$ a.e. $x\in \Omega$.
   \end{thm}



\section{Proof of Theorem~\ref{main}.}

We now give the proof of Theorem~\ref{main}, beginning with some preliminary lemmas. Throughout this proof we will assume that $\varphi$ is Lipschitz, as the modification for H\"older $\varphi$ is straightforward.

The next lemma is not optimal but is useful in allowing us to go from uniform contraction along stable manifolds upon returns to the base of the Young Tower (Property (P3) of~\cite{LY98}) to estimates of the contraction along stable leaves in the whole manifold. 

\begin{lemma}\label{lemma_technical}
Let $\beta_1$ be defined as in Section (2.1) Assumption (a) and $\beta_2$ be as in Theorem 2.2.
Let $D(m)=\{ (x,j)\in \Delta: |T^k W^s (x,j) | < {(\sqrt{\beta_1})}^{k} \mbox{ for all } k\ge m \} $.  Then for any $\delta >0$ there exists 
 $K(\delta) >0$ such that for all $m\ge K$, $\nu_{\Delta} (D (m)^c)\le C {(\beta_2+\delta)}^{m/2} $ for some constant $C>0$.
\end{lemma}

\begin{proof}
Let $\tau_r(x,j):=\# \{ k:  1< k \le r: F^k (x,j) \in \Lambda \}$, so that $\tau_r(x,j)$ denotes the number of times $k\in[1,r]$ that $F^k (x,j)$ lies in the base of the Young Tower.
Let $B_r=\{ (x,j) \in \Delta : \tau_r (x,j) \le \sqrt{r}\}$.
If $\tau_r (x,j) \le \sqrt{r}$ then there is at least one $k\in[0,r]$, such that $R(F^k (x,j)) >\sqrt{r}$ and hence
$B_r\subset \bigcup_{k=1}^r F^{-k} (R > \sqrt{r})$. Thus 
$\nu_{\Delta} ( B_r )\le r \nu(R>\sqrt{r}) < C_2 r {\beta_2}^{ \sqrt{r}}$.

Suppose now that $(x,j) \in B_r^c$. Then $|T^{r}W_s((x,j))|\le 2K
\beta_1^{ \sqrt{r}}$ by (a) and (c) and moreover
$\nu_{\Delta} (\bigcup_{r\ge m}  ( B_r ) )\le \sum_{r\ge m} C_2 r {\beta_2}^{ \sqrt{r}}$. Now the lemma follows from a 
straightforward calculation.

\end{proof}

\begin{cor}\label{cor_technical}
Lift $\varphi:  M \to \mR$ to $\varphi :\Delta \to \mR$ by defining $\varphi (x,j)=\varphi (T^j x)$.
Let $\beta_1$ be defined as in Section (2.1) (a).
Suppose $p\in D(m)=\{ (x,j)\in \Delta: |T^k W^s (x,j) | < {(\sqrt{\beta_1})}^{k} \mbox{ for all } k\ge m \} $ and let $L_n=[\frac{\ln n}{I(\alpha)}]$. Then if $q\in W^s (p)$,
$|S_{L_n} \varphi \circ F^m (p) - S_{L_n} \varphi \circ  F^m (q)|\le C \|\varphi \|_{\infty}  L_n  {\beta_1}^{m/2}$.
\end{cor}

\begin{proof}[Proof of Theorem~\ref{main}]
The main idea of the proof of Theorem~\ref{main} is to approximate functions on $\Delta$ by functions constant on stable manifolds, so that correlation decay 
estimates on the quotiented tower from \cite[Corollary 2.9]{KKM} can be used. 

We define an equivalence relation on $\Lambda$ by $z\sim x$ if $z\in W^s (x)$ and form the quotient space  $\overline{\Lambda}=\Lambda/\sim$
with corresponding partition $\{ \overline{\Lambda_j}\}$. The return time function $R: \overline{\Lambda}\to \N $ is well-defined (and the same in the quotiented and
unquotiented tower) as 
each  stable disk $W^s (x)$  lies in $\Lambda_j$ if $x\in \Lambda_j$ and $T^{R(j)}(W^s (x)) \subset W^s (T^{R(j)} x)$ for all $x\in \Lambda_j$.
So we have a well-defined induced map $\bar{f}: \overline{\Lambda} \to \overline{\Lambda}$.  We similarly define the 
quotient space of $\Delta$, denoted $\overline{\Delta}$. The separation time for $f: \overline{\Lambda}
\to \overline{\Lambda}$ extends to a separation time on $\overline{\Delta}$ by defining

\[  s((x,l),(y,l^{'})) =\left\{ \begin{array}{ll}
         s(x,y) & \mbox{if $l=l^{'}$};\\
        1 & \mbox{if $l \neq l^{'}$}.\end{array} \right. \]

We fix  $\beta_1$  from Section 2.1 Assumption (a) and define the metric $d_{\beta_1}$ on $\overline{\Delta}$ by $d_{\beta_1} (p,q)=\beta_1^{s(p,q)}$. Here we write
$p=(x,l)\in \overline{\Delta}$, $q=(y, l^{'})$.
We define the $\|\cdot\|_{\beta_1}$-norm by $\|\phi \|_{\beta_1}:=\|\phi\|_{\infty} + \sup_{p,q\in \Delta} \frac{|\phi (p)-\phi (q)|}{d_{\beta_1} (p,q)}$.
Functions $\phi$ and $\psi$ which are constant on stable manifolds in $\Delta$ naturally project to functions $\phi$ and $\psi$ (we use the same notation) on $\overline{\Delta}$ 
with the same $d_{\beta_1}$ Lipschitz constant and $L^{\infty}$ norm. If $\phi: \Delta \to \mR$
is constant on stable manifolds we  define the $\|\cdot\|_{\beta_1}$-norm by $\|\phi \|_{\beta_1}:=\|\phi\|_{\infty} + \sup_{p,q\in \Delta} \frac{|\phi (p)-\phi (q)|}{d_{\beta_1} (p,q)}$.

With this set-up the correlation estimate of ~\cite[Corollary 2.9]{KKM} can be stated:
\begin{prop}~\cite[Corollary 2.9]{KKM}

Suppose that $\phi,~\psi: \Delta \to \mR$  are constant on stable manifolds then for some constants $C$, $\beta_3\in (0,1)$,
\[
| \int_{\Delta}  \phi (\psi\circ F^j)\, d\nu_{\Delta} -  \int_{\Delta}  \phi \, d\nu_{\Delta}  \int_{\Delta}  \psi \, d\nu_{\Delta}|\le C\|\phi\|_{\beta_1} \|\psi\|_{\infty} \beta_3^j
\]
for all $j\ge 0$ .

\end{prop}

In the case that $\varphi$ is not an $L^1$ coboundary i.e. there exists no $\psi$ such that
$\varphi =\psi \circ T -\psi$, $\psi \in L^1 (m)$ it has been shown~\cite{Nic,Reybellet_Young} under the assumptions
of Theorem~\ref{main} that $\varphi$ has exponential large deviations with a rate function $I(\alpha)$. Thus assumption (a)
of Proposition~\ref{prop:erdos1} holds and we therefore only need to prove 
$\mu (\{\max_{0\le m\le n-L_n} S_{L_n}\circ T^m \le L_n(\alpha-\epsilon)\})$ is summable
in order to get the lower bound by an application of the Borel-Cantelli lemma.   
This direction is more difficult and uses
differential and dynamical information on the system.


For the reader's convenience we recall our assumptions:

\noindent {\it Assumptions:} there exist constants $K\ge 1$ and $0<\beta_1<1$ such that 

(a) if $z \in W^s (x)$ then $d(f^n z, f^n x) \le K \beta_1^n$;

(b)  if $z\in W^u (x)$ then  $d(f^n z, f^n x)\le K \beta_1^{s(x,z)-n}$;

(c) if $z,x\in \Lambda$ then $d(T^j z,T^j x) \le K (d(z,x)+d(fz, fx))$ for all $0\le j \le \min \{ R(z), R(x)\}$.

We lift $\varphi$ from $M$ to $\Delta$ by defining $\varphi (x,j)=\varphi (T^j x)$. We will use the same notation for 
$\varphi$ on $\Delta$ as we use for $\varphi$ on $M$. 


To simplify notation we will sometimes write $p=(x,j)$ for a point $p\in \Delta$.

For $0<\epsilon \ll \alpha$ put 
$$
A_n(\epsilon):=\{ (x,j) \in \Delta: S_{L_n} \le L_n (\alpha-\epsilon)\},
$$
where 
$$
S_n (x,j)=\sum_{k=0}^{n-1} \varphi \circ F^k (x,j)
$$
is the $n$th ergodic sum of $\varphi$.
Define
$$
B_n (\epsilon)=\bigcap_{m=0}^{n-L_n} F^{-m}A_n(\epsilon)
=\left\{ (x,j) \in \Delta: \max_{0\le m\le n-L_n}S_{l_n}\circ F^m \le L_n (\alpha -\epsilon)\right\}.
$$

The theorem follows by the Borel-Cantelli lemma once we show that 
$\sum_{n=1}^{\infty} \nu_{\Delta} (B_n(\epsilon)) <\infty$.

To do this we  will use a blocking argument to take advantage of decay of correlations and 
intercalate by blocks of length $\kappa_n:=\ln^{\kappa}(n)$,  where $\kappa$ will be specified later.

For $1\le j <r_n:= [\frac{n}{\kappa_n}]$ put 
\[
E_n^j (\epsilon):=\bigcap_{m=1}^{j} F^{-m[\kappa_n]}A_n(\epsilon)
\]
which is a nested sequence of sets. Note that 
$\nu_{\Delta}  (B_n (\epsilon) )\le \nu_{\Delta} (E_n^{r_n} (\epsilon) )$. 

We also have the recursion
\[
E_n^{j}(\epsilon)=A_n (\epsilon) \cap F^{-\kappa_n}E_n^{j-1}(\epsilon)
\]
 $j=1,\dots ,r_n$, which implies
 \[
 \nu_{\Delta} (E_n^{j}(\epsilon))=\nu_{\Delta} (A_n (\epsilon) \cap F^{-\kappa_n}E_n^{j-1}(\epsilon) )
 \]

Recall $D(m)=\{ (x,j) \in \Delta : |T^k W^s (x) | <(\sqrt{\beta_1})^{k} \mbox{ for all } k\ge m \}$. 
Hence given $\delta >0$ such that $\beta_2^{'}:=\beta_2+\delta<1$ by  Lemma~\ref{lemma_technical}
we  may estimate $\nu_{\Delta} (D(\kappa_n)^c) \le (\beta_2^{'})^{\kappa_n/2}$ for sufficiently large $n$.

Furthermore if $m\ge \kappa_n$,  $p\in D(m)$ and $q\in W^s (p)$ then $|S_{L_n}\circ F^m(p)-S_{L_n}\circ F^m(q)|\le C\|\varphi\|_{\infty}L_n \beta_1^{\kappa_n/2}$ 
by the corollary to Lemma~\ref{lemma_technical}. We will take $\kappa$ and $n$ large enough that $C\|\varphi\|_{\infty}L_n \beta_1^{\kappa_n/2}<\frac{\epsilon}{2}$.

 Accordingly for 
large $n$ if $m\ge \kappa_n$,  $p\in D(m)\cap F^{-m} A_n (\epsilon) $ and $q\in W^s (p)$ then $F^m q \in A_n (\frac{\epsilon}{2} )$.

\vspace{.5cm}

\noindent {\it First Approximation.}

We now approximate $1_{A_n (\epsilon)\cap D(\kappa_n)}$ by a function $g_n^{\epsilon}$ which is constant on stable manifolds  by
requiring that if $p \in A_n (\epsilon) \cap D(\kappa_n)$ then $g_n^{\epsilon} (p)=1$ on $W^s(p)$ and $g_n^{\epsilon}=0$ otherwise. 
Thus $  \{ g_n^{\epsilon} =1\} \subset A_n (\frac{\epsilon}{2})$
and
\[
\nu_{\Delta} (g_n=1)\le \nu_{\Delta} (A_n (\frac{\epsilon}{2}))
\]
Furthermore
\[
A_n (\epsilon) \subset \{ g_n^{\epsilon}=1\} \cup D(\kappa_n)^c
\]
hence
\[
\nu_{\Delta} (A_n (\epsilon)) \le \nu_{\Delta} (g_n^{\epsilon}=1) + \nu_{\Delta}  (D(\kappa_n)^c).
\]
For $j=1,\ldots, r_n$ let
\[
G_n^j (\epsilon)=:\prod_{i=1}^j g_{n}^{\epsilon} \circ F^{i[\kappa_n]}
\]
and note $\nu_{\Delta}  (E_n^j (\epsilon)) \le \nu_{\Delta} (G_n^j (\epsilon)) + j \nu_{\Delta}  (D(\kappa_n)^c)$.

\vspace{.5cm}

\noindent {\it Second Approximation}

We will approximate $g_n^{\epsilon}$  (considered as a function on $\overline{\Delta}$) by
a $d_{\beta_1}$ Lipschitz function $h_n^{\epsilon}$  which extends to a function on $\Delta$ by 
requiring $h_n^{\epsilon}$ to be  constant on stable
manifolds.

 First define 
$$
h_n^{\epsilon} (\bar{p}):=\max \{0, 1-d_{\beta_1}(\bar{p},\mbox{supp} (g_n^{\epsilon} )) \beta_1^{-\sqrt{\kappa_n}}\}
$$
on $\overline{\Delta}$ and then extend so that it is constant on local stable manifolds and hence is a function on $\Delta$.
In particular $h_n^{\epsilon}$ has support in points such that $d_{\beta_1} ( p, \mbox{supp}  (g_n^{\epsilon} ))\le \beta_1^{\sqrt{\kappa_n}}$ and
$\|h_n^{\epsilon}\|_{\beta_1}\le \beta_1^{\sqrt{\kappa_n }}$ by~\cite[Section 2.1]{Stein}.

By $(b)$ and $(c)$  if $z\in W^{u} (p)$ and $d_{\beta_1} (p,z)< \beta_1^{\sqrt{\kappa_n }}$ then $d(F^j p, F^j z)\le 2K \
\beta_1^{\sqrt{\kappa_n}-L_n}$ for all $j\le L_n$.

Hence if $d_{\beta_1} (z,\mbox{supp} (g_n^{\epsilon} ))\le \beta_1^{\sqrt{\kappa_n}}$ then there exists $p\in \mbox{supp} (g_n^{\epsilon} )$ such that $d(F^j p, F^j z)\le 2K \beta_1^{\sqrt{\kappa_n}-L_n}$ for all $j\le L_n$ and 
hence 
\[
|  \sum_{j=0}^{L_n} [\varphi \circ F^j (z) - \varphi \circ F^j (p)]|\le C L_n \beta_1^{\sqrt{\kappa_n}-L_n }\le \frac{\epsilon}{2}
\]
for sufficiently large $n$.  This implies  that $\nu_{\Delta} (g_n^{\epsilon} )\le \nu_{\Delta} (h_n^{\epsilon}) \le \nu_{\Delta} (A_n (\frac{\epsilon}{2}))$.

As $h_n^{\epsilon} $ Lipschitz  in the $d_{\beta_1}$ metric we  obtain by  Proposition~6.3 
\begin{eqnarray*}
\nu_{\Delta}  (E_n^{j} (\epsilon) ) 
& \le & \int_{\Delta}  (G_n^{j} (\epsilon))\,d\nu_{\Delta} +j\nu_{\Delta} (D(\ln^{k}( n)^c)\\
&\le &\int_{\Delta} (g_n^{\epsilon}  \cdot G_n^{j-1} \circ F^{\kappa_n} )\,d\nu_{\Delta}  + Cn {(\beta_1^{'})} ^{\kappa_n/2} \\
&\le &\int  h_n^{\epsilon}\,d\nu_{\Delta}  \int G_n^{j-1} (\epsilon)\, d\nu_{\Delta}
+c_3 \beta_3^{\kappa_n} \|h_n^{\epsilon}\|_{\beta_1} \|G_n^{j-1} (\epsilon)\|_{\infty} +Cj {(\beta_1^{'})}^{\kappa_n/2} \\
&\le&  \nu_{\Delta}(A_n (\frac{\epsilon}{2})) \nu_{\Delta} (G_n^{j-1}(\epsilon) )+c_3 \beta_3^{\kappa_n}\beta_1^{-\sqrt{\kappa_n}}  +Cj {(\beta_1^{'})} ^{\kappa_n/2}. 
\end{eqnarray*}
Iterating this estimate yields
$$
\nu_{\Delta}  (E_n^0 (\epsilon) ) 
\le \nu_{\Delta} ( A_n (\frac{\epsilon}{2}) )^{[n/\kappa_n]} +n c_3 \beta_3^{\kappa_n} \beta_1^{-\sqrt{\kappa_n}}
+n^2C\beta_1^{\kappa_n/2}.
$$
The terms $n c_3 \beta_3^{\kappa_n} \beta_1^{-\sqrt{\kappa_n}}$ and $n^2C\beta_1^{\kappa_n}$ are  summable if we take
$\kappa >3$ in the definition of $\kappa_n$.


In order to verify summability of the  
$\nu_{\Delta} ( A_n (\frac{\epsilon}{2} ))^{[n/\kappa_n]}$ term
we proceed as in the proof of Proposition~\ref{prop:erdos1}
using large deviations. 
By the existence of a rate function we obtain
$\nu_{\Delta}((A_n (\frac{\epsilon}{2}))^c)\ge e^{-L_n(I(\alpha-\frac{\epsilon}{2})+\delta_1)}$ for some $0<\delta_1$
and hence $1-\nu_{\Delta}   (A_n (\frac{\epsilon}{2}))\ge e^{-L_n(I(\alpha-\frac{\epsilon}{2})+\delta_1)}$
for some $0<\delta_1$.  
Hence   $\nu_{\Delta}   (A_n (\frac{\epsilon}{2}))\le 1-n^{-\rho}$ where
$\rho=\frac{I(\alpha-\frac{\epsilon}{2})}{I(\alpha)}+\delta_1$ is less than $1$ for $\delta_1>0$ small enough.
The principal term can  be bounded by
$$
 \nu_{\Delta}( A_n (\frac{\epsilon}{2}))^{[n/\kappa_n]}
 \le (1-n^{-\rho})^{[n/\kappa_n]}
  $$
 which is also summable over $n$.
Hence by Borel-Cantelli we conclude that the set $\{ B_n(\epsilon) \mbox{ i.o.}\}$ has 
measure zero. This concludes the proof.

\end{proof}


\section{Erd\"os-R\'enyi laws for  Young Towers with polynomial tails.}\label{sec:intro}

We now consider Young Towers with polynomial tails in the sense that $\nu_{\Delta} (R>n) \le Cn^{-\beta}$.

\subsection{Upper bounds.}

We first prove a  general result. We suppose that $(T, M, \mu)$ is an ergodic dynamical
system and $\varphi: M \to \mathbb R$ is a bounded observable. We assume also
$$
\mu \!\left(\left|\frac{1}{n}S_n (\varphi)-\bar\varphi\right| > \epsilon\right)\le C(\epsilon)  n^{-\beta}.
$$

\begin{thm}\label{poly1} Assume that $\bar\varphi=\mu(\varphi)=0$, $\varphi$ is bounded  and  for every $\epsilon>0$ there exists a  constant $C(\epsilon)>0$ and 
$\beta>1$ so that 
$$
\mu \!\left(\left|\frac{1}{n}S_n (\varphi)\right| > \epsilon\right)\le C (\epsilon)  n^{-\beta}.
$$
Then if  $\tau>\frac{1}{\beta}$ for $\mu$ a.e. $x\in M$, 
\[
  \lim_{n\to\infty} \max_{0\le m\le n-n^{\tau}} n^{-\tau}S_{n^{\tau}}\circ T^m(x)  = 0.
\]
\end{thm}

\begin{proof} Choose  $\tau>\frac{1}{\beta}$ and put $L_n=n^{\tau}$.
Let $\epsilon>0$ and define 
\[
A_n:=\{ x\in X:  \max_{0\le m\le n-L_n} |S_{L_n}\circ T^m | \ge L_n \epsilon\}.
\]
Then $\mu (A_n)\le n \mu (S_{L_n} \ge \epsilon L_n)\le c_1 (\delta)   n^{1-\tau\beta }=c_1 n^{-\delta}$,
for some $c_1 >0$, where $\delta=\tau\beta -1$.

Let $p>\frac{1}{\delta}$ (i.e.\ $\delta p>1$) and consider the subsequence $n=k^{p}$. 
Since $\sum_k\mu(A_{k^p})\le c_1\sum_kk^{-p\delta}<\infty$, we obtain via the Borel-Cantelli lemma
that for $\mu$ a.e. $x \in X$ 
\[
 \limsup_{k\to\infty} \max_{0\le m\le k^p-L_{k^p}} L_{k^p}^{-1}|S_{L_{k^p}}\circ T^m| \le \epsilon.
\]
To fill the gaps use that $k^p-(k-1)^p=O(k^{p-1})$ and we obtain (as $\varphi$ is bounded) that
\[
\frac{S_{L_{k^p}}\circ T^m }{L_{k^p}}=\frac{S_{L_{(k-1)^p}}\circ T^m}{L_{k^p}}+\mathcal{O}\!\left(\frac{1}{k}\right)
\]
where the implied constant is uniform in $x\in X$as $\varphi$ is bounded.
As $\lim_{k\to \infty}\frac{k^p}{(k-1)^{p}}=1$ we conclude 
\[
\lim_{k\to \infty} \frac{|S_{L_{k^{p}}}|}{L_{k^{p}}}=\lim_{k\to \infty} \frac{|S_{L_{(k-1)^{p}}}|}{L_{k^p}}.
\]
Since any $n\in \mathbb N$ satisfies $(k-1)^p \le n\le k^p$ for some $k$ and $\varphi$ is bounded, it follows that 
\[
 \limsup_{n\to \infty} \max_{0\le m\le n-L_{n}} |S_{L_{n}}\circ T^m |/L_{n} \le \epsilon.
\]
As $\epsilon$ was arbitrary this gives the upper bound. 
\end{proof}




\subsection{Lower bounds.}

Now we suppose there exists $\gamma \ge \beta$, an observable $\varphi$ and an $\alpha>0$ such that for all $n$,
$\mu \!\left(\left|\frac{1}{n}S_n (\varphi)-\bar\varphi\right| > \alpha \right)\ge C(\alpha) n^{-\gamma}$.
We show if we take a window of length $n^{\tau}$, $\tau < \frac{1}{1+\frac{\beta+1}{\beta}\gamma}$ then the time-averaged
fluctuation persists almost surely. In the case that $\gamma$ limits to $\beta$ then we require $\tau < \frac{1}{2+\beta}$. Comparing
Theorem~\ref{poly1} and Theorem~\ref{poly2} there is a gap $\frac{1}{1+\frac{\beta+1}{\beta}\gamma}<\tau <\frac{1}{\beta}$
for which we don't know the almost sure limit of windows of length $n^{\tau}$. In Example~\ref{example1} we show that $\tau<\frac{1}{\beta+1}$
is required to ensure that a time-averaged
fluctuation persists almost surely.

\begin{thm}\label{poly2}

Suppose that $(T,M,\mu)$ is modeled by a Young Tower and $\bar{\nu}_{\Delta} (R>n)\le C n^{-\beta }$.
Suppose that $\gamma \ge \beta$ and there exists a function $C$ which is continuous on a neighborhood of  $\alpha>0$ such that
\[
\mu \!\left(\left|\frac{1}{n}S_n (\varphi)-\bar\varphi\right| > \alpha \right)\ge C(\alpha) n^{-\gamma}
\]

Then if  $0<\tau< \frac{1}{1+ \gamma\frac{\beta+1}{\beta}}$ for $\mu$ a.e. $x\in M$
\[
  \lim_{n\to\infty} \max_{0\le m\le n-n^{\tau}} n^{-\tau}S_{n^{\tau}}\circ T^m(x)  \ge \alpha
\]

\end{thm}

\begin{proof}
Let $0<\epsilon \ll \alpha$ and put
\[
A_{n^{\tau}} (\epsilon) =\{ (x,j): \sum_{r=1}^{n^{\tau}} \varphi \circ F^r (x,j) \le \alpha -\epsilon \}.
\]
Since $\varphi$ is Lipschitz continuous  with Lipschitz constant $L$, then if $y\in A_{n^\tau}(\epsilon)$ 
and $d(y,y')<\frac{\epsilon}{2Ln^\tau}$, then $y'\in A_{n^\tau}(\epsilon/2)$. Hence let us choose $n_1$ 
so that $K\beta_1^{n_1}<\frac\epsilon{2Ln^\tau}$ and define
$$
B_{n^{\tau} }(\epsilon) 
=\{ (x,0)\in\Lambda: \exists\:0\le  j< R(f^{n_1}x )~\mbox{ with } (f^{n_1}x,j) \in A_{n^{\tau}} (\epsilon)\}
=f^{-n_1}(\pi A_{n^\tau}(\epsilon)),
$$
where $\pi:\Delta\to\Lambda$ is the projection given by $\pi((x,j))=(x,0)$ ($j<R(x)$).
The choice of the integer achieves that if $(x,0) \in B_{n^{\tau} }(\epsilon)$ and $(x',0)\in W^s (x,0)$ then 
 $(f^{n_1} x' ,0)  \in\pi( A_{n^{\tau}} (\frac{\epsilon}{2}))$. This is a consequence of Assumption (a).
By assumption 
\[
\nu_{\Delta} (A_{n^{\tau}} (\epsilon) )\ge C (\alpha-\epsilon) n^{-\gamma \tau}.
\]
For $\delta > \frac{\tau \gamma}{\beta}$ we have 
 \[
 \nu_{\Delta} ( R> n^{\delta})=o(n^{-\delta})
 \]
 as by assumption $ \nu_{\Delta} ( R> \ell)\le C\ell^{-\beta}$.
 
 Since $\nu_{\Delta} =\bar{\nu}\times \mbox{ (counting measure) }$ we get for $D\subset\Delta$
 $$
 \bar\nu(\pi(D))
 \ge \frac{\nu_\Delta(D)-\nu_\Delta(R>n^\delta)}{n^\delta}.
 $$
 Consequently
 $$
 \bar{\nu} (\pi(A_{n^{\tau}} (\epsilon)))
  \ge \left(C(\alpha-\epsilon) n^{-\tau\gamma}-o(n^{-\delta\beta})\!\right)n^{-\delta}
   $$
   and since $\delta\beta>\tau\gamma$ the first term dominates and we obtain
   $$
    \bar{\nu} (\pi(A_{n^\tau} (\epsilon)))
  \ge c_1n^{-\tau\gamma-\delta}
  $$
  for some $c_1>0$ and since $f^{n_1}$ preserves $\bar{\nu}$,
 $$
 \bar{\nu} (B_{n^{\tau}} (\epsilon))\ge c_1n^{-\tau\gamma-\delta}.
 $$
 
 We can now define 
 $$
\tilde{B}_{n^\tau}(\epsilon)=\bigcup_{x\in B_{n^\tau}(\epsilon)}W^s(x)
$$
which by choice of $n_1$ implies that 
$$
 \tilde{B}_{n^\tau}(\epsilon)\subset B_{n^\tau}(\epsilon/2).
 $$
   We now approximate $1_{B_{n^{\tau}}} (\epsilon)$ by a function $h_{n^{\tau}} (\epsilon)$ 
 which has Lipschitz constant $\beta_1^{-n^\tau}$ in the $d_{\beta_1}$-norm, that is we define
 \[
 h_{n^{\tau}} (\epsilon) (p) = \max (0,1- d(p, B_{n^{\tau}} (\epsilon) \beta_1^{-n^{\tau}})
 \]
 where we write $p$ for $(p,0)$. We can choose $n_1$ to be much smaller than $n^\tau$ 
 and therefore, since by Assumption (b) and (c) if $d(p, B_{n^{\tau}} (\epsilon))<  \beta_1^{\tau_n}$ then
 $d(f^{n_1}p,B_{n^\tau}(\epsilon))<K\beta_1^{n^\tau-n_1}<\frac\epsilon{2Ln^\tau}$ which 
 implies that the support of $h_{n^{\tau}} (\epsilon)$ is contained in $B_{n^{\tau}} (\epsilon/2)$.
 
 Now we let $\tau_1>\tau$ but  $\tau_1-\tau < 1- (\tau\gamma\frac{\beta+1}{\beta} +\tau)$  and consider 
 \[
 G_n(\epsilon)=\bigcap_{m=0}^{[n/n^{\tau_1}]}f^{-mn^{\tau_1}}B_{n^{\tau}} (\epsilon)
 \]
 We will show that 
 \[
 \sum_n \bar{\nu} ( G_n(\epsilon)) <\infty
 \]
 
 Now
\begin{eqnarray*}
 \bar{\nu} (G_n(\epsilon))
 &\le& \bar{\nu} \!\left(\prod_{m=0}^{n^{1-\tau_1}} h_{n^{\tau}} (\epsilon) \circ f^{mn^{\tau_1}}\!\right)\\
& \le &\bar{\nu} (h_{n^{\tau}} (\epsilon)) \bar{\nu}(G_{n-1}(\epsilon))
 +c_3\|h_{n^{\tau}} (\epsilon)\|_{\beta_1}\|  |G_{n-1} (\epsilon) |_{\infty}\beta_3^{n^{\tau_1}}\\
& \le& [\bar{\nu}(h_{n^{\tau} }(\epsilon))]^{n^{1-\tau_1}}+ nC_3\beta_3^{n^{\tau_1}}\beta_1^{-n^{\tau}}
 \end{eqnarray*}
 The term $nC_3\beta_3^{n^{\tau_1}}\beta_1^{-n^{\tau}}$ is summable in $n$ as $\tau_1>\tau$. 	The principal term is estimated by
$$
[\bar{\nu}(h_{n^{\tau} }(\epsilon))]^{n^{1-\tau_1}}
\le \left(1- C(\alpha-\frac{\epsilon}{2}) n^{-\gamma\tau-\delta}\!\right)^{n^{1-\tau_1}}
\le \exp\!\left(-C(\alpha-\epsilon/2)n^{1-\tau_1-\gamma\tau-\delta}\!\right)
$$
Since $\tau_1>\tau$ can be chosen arbitrarily close to $\tau$ and $\delta>\frac{\tau\gamma}\beta$ 
can be chosen to achieve the power $1-\tau-\tau\gamma-\delta$ is positive for any chosen
$\tau<(1+\gamma\frac{\beta+1}\beta)^{-1}$ we obtain that the principal terms are summable which 
implies summability of $\bar\nu(G_n(\epsilon)$.

Now define 
$$
E_n:=\{ (x,0) : \mbox{for all } j< n : \sum_{r=0}^{n^{\tau} }\varphi (F^{R_{n_1}(x)+j +r} x,0)\le (\alpha-\frac{\epsilon}{2})  n^{\tau} \},
$$
where $R_\ell =\sum_{i=0}^{\ell-1}R\circ f^i$ is the $\ell$-th ergodic sum of $R$.
As $E_n (\epsilon) \subset G_n (\epsilon)$, $\bar\nu(G_{n} (\epsilon))$  summable implies  that 
$\sum_{n=1}^{\infty} \bar{\nu} (E_n (\epsilon))<\infty$.

By Birkhoff's ergodic theorem
\[
\lim_{n\to \infty} \frac{R_n (x,j)}{n}=\bar{R}=\frac1{\nu_{\Delta} (\Lambda)}
\]
for $\nu_{\Delta}$ a.e. $(x,j)\in \Delta$, and so the theorem follows.

\end{proof}

\begin{examp}\label{example1}
The condition  $\tau <  \frac{1}{1+ \gamma\frac{\beta+1}{\beta}}$ is close to  optimal in that, taking $\gamma=\beta$,  we require
$\tau <\frac{1}{2+\beta}$. We may construct a Young Tower and observable $\varphi$, $\int_{\Delta} \varphi \,d\nu_{\Delta}=0$ and $\alpha>0$ such that  
$\nu_{\Delta} (S_{n_{\tau} } \varphi (x,j) \ge n^{\tau} \alpha) \le C n^{-\tau \beta}$, yet for all $\tau >\frac{1}{\beta+1}$,
\[
  \lim_{n\to\infty} \max_{0\le m\le n-n^{\tau}} n^{-\tau}S_{n^{\tau}}\circ T^m =0
\]

 We sketch the main idea of the tower and observable  and make a couple of technical adjustments to ensure the tower is mixing and that the 
 observable is not a coboundary.  The construction is based on that of~\cite{Bryc}.
The base  partition consists of disjoint 
intervals $\Lambda_i$ of length $i^{-\beta-2}$ and height $2i$. Above the base element $\Lambda_i$ the levels of the tower consist of 
$\{ (x,j): 0\le j \le 2i-1\}$. We define $\varphi$ on the Tower by, if $x\in \Lambda_i$,
\[ \varphi (x,j) = \left\{ \begin{array}{ll}
         -1 & \mbox{if $0\le j < i$};\\
        1 & \mbox{if $i \le  j  < 2i$}.\end{array} \right. \] 
Clearly $\nu_{\Delta} (\varphi)=0$.

Let $0< \alpha< 1$. Note that $S_{n_{\tau} } \varphi (x,j) \ge n^{\tau} \alpha$ only if $(x,j) \in (R>n^{\tau})$ and in fact $\nu_{\Delta} (S_{n_{\tau} } \varphi (x,j) \ge n^{\tau} \alpha)
\ge C \nu_{\Delta} (R>2n^{\tau} )=\sum_{r=2
n^{\tau}}^{\infty} (2j)j^{-2-\beta}\le C n^{-\tau \beta}$.

However if $\tau>\frac{1}{\beta+1}$ then $\sum_{j\ge n^{\tau}} \bar{\nu} (\Lambda_j)\le \sum_{n=1}^{\infty} n^{-\tau(\beta+1)}<\infty$. Hence by the Borel-Cantelli lemma
$f^n (x, 0)\in \bigcup_{j>n^{\tau} } \Lambda_j$ only finitely many times for $\bar{\nu}$ a.e. $(x,0)$. This implies that for $\bar{\nu}$ a.e. $(x,0)$ there exists an $N(x)$ such that for
all $n \ge N(x)$
$$
 \mbox{for all  } j< n : \sum_{r=0}^{n^{\tau}} \varphi (f^{j+r} x,0)< \alpha  n^{\tau}. 
$$
Hence for  $\mu$ a.e. $x\in M$
\[
  \lim_{n\to\infty} \max_{0\le m\le n-n^{\tau}} n^{-\tau}S_{n^{\tau}}\circ T^m < \alpha
\]
for every $\alpha>0$.

The same argument shows for   $\nu_{\Delta}$ a.e. $(x,j)$
\[
  \lim_{n\to\infty} \max_{0\le m\le n-n^{\tau}} n^{-\tau}S_{n^{\tau}}\circ T^m =0
\]
and
\[
 \lim_{n\to\infty} \min_{0\le m\le n-n^{\tau}} n^{-\tau}S_{n^{\tau}}\circ T^m =0
\]

The heights of the levels in the tower above are all multiples of $2$. 
Furthermore the observable  $\varphi$ is a coboundary. If we define
\[
\psi  (x, j) = \begin{cases}j
 & \mbox{ if }  x\in \Lambda_{k},0\le  j \le k\\
2k-j &\mbox{ if } x\in \Lambda_{k}, k < j \le 2k-1\end{cases}.
\]

It is easy to check that
 \[
\varphi=\psi\circ F -\psi
\]

We will modify the tower and the observable so that the greatest common denominator of the
return time function $R$ is $1$ (to ensure the tower is mixing) and that the new observable is not a coboundary. We change  $\Lambda_3$ to have height $3$. This entails that the 
tower is mixing. On  the levels above $\Lambda_3$ we modify $\varphi$ to $\varphi_1$
so that $\varphi_1 (x,j)=\kappa>0$, $j=0,1,2$, $x\in \Lambda_3$ where $\kappa>0$ is small but $\varphi_1=\varphi$ elsewhere .  This entails $r_1:=\nu_{\Delta} (\varphi_1)=\kappa \nu_{\Delta}(\Lambda_3) >0$. We subtract $r_1/(\nu_{\Delta}(\Lambda_2))$
from the value of $\varphi_1$ on $\Lambda_2$ to form a new observable $\varphi_2$ such that $\nu_{\Delta} (\varphi_2)=0$.
 Since $F^3$ has a fixed point $p$ on $\Lambda_3$ and since $\sum_{j=0}^2 \varphi_2(x,j)\not = 0$ we conclude $\varphi_2$ is not a coboundary (by the Li\v{v}sic theorem~\cite{Nicol_Scott}).  The new tower with observable
 $\varphi_2$ we defined 
 has  the properties of the former pertinent to our example.

\end{examp}

\end{document}